\documentclass[11pt]{amsart}

\usepackage{amssymb}
\usepackage{eucal}
\usepackage{lmodern}
\usepackage{color}
\usepackage[T1]{fontenc}

\DeclareSymbolFont{SY}{U}{psy}{m}{n}
\DeclareMathSymbol{\emptyset}{\mathord}{SY}{'306}

\theoremstyle{plain}

\newtheorem{thm}{Theorem}[section]
\newtheorem{cor}[thm]{Corollary}
\newtheorem{lem}[thm]{Lemma}
\newtheorem{prop}[thm]{Proposition}
\theoremstyle{definition}
\newtheorem{defn}[thm]{Definition}
\newtheorem{rem}[thm]{Remark}
\newtheorem{ex}[thm]{Example}

\numberwithin{equation}{section}

\def\C{{\mathbb C}}

\def\N{\mathbb{N}}

\def\ra{\rightarrow}
\def\ov{\overline}

\def\a{\alpha}

\def\g{\gamma}
\def\d{\sum}

\def\wi{\widetilde}

\def\beq{\begin{eqnarray}}
\def\eeq{\end{eqnarray}}
\def\beqa{\begin{eqnarray*}}
\def\eeqa{\end{eqnarray*}}
\def\wt{\widetilde}

\def\T{\boldsymbol T}
\def\M{\boldsymbol M}

\begin{document}
\title[Infinite divisibility and curvature inequality]{Infinitely divisible
metrics and curvature inequalities\\ for operators in the Cowen-Douglas class}
\author[Biswas]{Shibananda Biswas}
\address[Biswas]{Department of Mathematics, Ben Gurion University of the Negev,
Be'er Sheva 84105, Israel}
\email{shibananda@gmail.com}
\author[Keshari]{Dinesh Kumar Keshari}
\address[Keshari]{Department of Mathematics, Indian Institute of Science,
Bangalore 560012, India} \email{kesharideepak@gmail.com}
\author[Misra]{Gadadhar Misra}
\address[Misra]{Department of Mathematics, Indian Institute of Science,
Bangalore 560012, India}
\email{gm@math.iisc.ernet.in}
\thanks{The work of S. Biswas was supported by a postdoctoral fellowship
funded by the Skirball Foundation via the Center for Advanced
Studies in Mathematics at Ben-Gurion University of the Negev. The
work of D.K. Keshari was supported in the form of a CSIR Research
Fellowship at the Indian Institute of Science. The work of G.
Misra was  supported in part by UGC - SAP and by a grant from the
Department of Science and Technology.} \subjclass[2000]{47B32,
47B35} \keywords{Cowen-Douglas class, curvature inequality,
holomorphic Hermitian vector bundle, kernel function, infinite
divisibility}
\begin{abstract}
The curvature $\mathcal K_T(w)$ of a contraction $T$ in the Cowen-Douglas
class $B_1(\mathbb D)$ is bounded above by the curvature $\mathcal K_{S^*}(w)$
of the backward shift operator. However, in general, an operator satisfying the
curvature inequality need not be contractive. In this note, we characterize a
slightly smaller class of contractions using a stronger form of the curvature
inequality.  Along the way, we find conditions on the
metric of the holomorphic Hermitian vector bundle $E_T$ corresponding to the
operator
$T$ in the Cowen-Douglas class $B_1(\mathbb D)$ which ensures negative definiteness of
the curvature function. We obtain a generalization for commuting tuples of
operators in the class $B_1(\Omega)$, for a bounded domain $\Omega$ in
$\mathbb C^m$.
\end{abstract}

\maketitle
\section{Introduction}
Let $\mathcal H$ be a complex separable Hilbert space and $\mathcal
L(\mathcal H)$ denote the collection of bounded linear operators on
$\mathcal H$. The following important class of operators was introduced
in \cite{cd}.

\begin{defn}
For a connected open subset $\Omega$ of $\mathbb C$ and a
positive
integer $n$, let
\begin{eqnarray*}
B_n(\Omega)  =  &\big \{& T\in\mathcal L(\mathcal H)\,|\,\,
\Omega\subset\sigma(T),\\
&& {\mathrm{ran}}\,(T-w)= \mathcal H\mbox{ for }w\in\Omega, \\
&&\bigvee_{w\in\Omega}\ker(T-w)= \mathcal H,\\
&&\dim~\ker(T-w)= n\mbox{ for } w\in\Omega\,\,  \big \},
\end{eqnarray*}
where $\sigma(T)$ denotes the spectrum of the operator $T$.
\end{defn}

We recall (cf. \cite{cd}) that an operator $T$ in the class $B_n(\Omega)$
defines a holomorphic Hermitian vector bundle $E_T$ in a natural manner. It is
the sub-bundle of the trivial bundle $\Omega\times\mathcal H$ defined by
$$E_T= \{(w, x)\in\Omega\times\mathcal H: x\in \ker(T-w)\}$$
with the natural projection map $\pi:E_T\to \Omega$,  $\pi(w, x)=
w$. It is shown in \cite[Proposition 1.12]{cd} that the mapping
$w\longrightarrow \ker(T-w)$ defines a rank $n$  holomorphic
Hermitian vector bundle $E_T$ over $\Omega$ for $T\in
B_n(\Omega)$. In \cite{cd}, it was also shown that the equivalence
class of the holomorphic Hermitian vector bundle $E_T$ and the
unitary equivalence class of the operator $T$ determine each
other.

\begin{thm} \label{eq}
The operators $T$ and $\wi T$ in $B_n(\Omega)$ are unitarily equivalent
if and only if the corresponding holomorphic Hermitian vector bundles $E_T$ and
$E_{\wi T}$ are equivalent.
\end{thm}

In general, it is not easy to decide if two holomorphic Hermitian vector
bundles are equivalent except when the rank of the bundle is $1$. In this
case, the curvature
\begin{eqnarray*}\mathcal
K(w)= - \frac{~\partial^2\log\parallel{\gamma(w)}\parallel^2}{\partial{w}
\partial{\overline{w}}},
\end{eqnarray*}
of the line bundle $E$, defined with respect to a non-zero holomorphic section
$\g$ of $E$, is a complete invariant. The definition of the curvature is
independent of the choice of the section $\g$:  If $\g_0$ is another holomorphic
section of $E$, then
$\g_0=\phi\g$ for some non-zero holomorphic function $\phi$ defined on an
open subset $\Omega_0$ of $\Omega$, consequently the harmonicity of
$\log\,|\phi|$, completes the verification.

Let $T\in B_1(\Omega)$. Fix $w\in\Omega$ and let $\g$ be a
holomorphic section of the line bundle $E_T$. From \cite[Lemma
1.22]{cd}, it follows that the vectors $\g(w)$ and $\partial\g(w)$
from a basis of $\ker(T-w)^2$. Let $N_{T}(w) = T|_{\ker(T-w)^2}$
and $\{\g_1(w), \g_2(w)\}$ be the basis obtained by applying
Gram-Schmidt ortho-normalization to the vectors $\g(w)$ and
$\partial\g(w)$. The linear transformation $N_{T}(w)$ has the
matrix representation
$$
N_T(w) = \begin{pmatrix}
w & h_T(w)\\
0 & w
\end{pmatrix},
$$
where $h_T(w) = \big(-\mathcal K_{T}(w)\big)^{-\frac{1}{2}}$, with respect to
the orthonormal basis $\{\g_1(w),\g_2(w)\}$.

The curvature $\mathcal K_{T}(w)$ of an operator $T$ in $B_1(\Omega)$ is
negative. To see this, recall that  the curvature
may also be expressed (cf.\cite[page - 195]{cd}) in the form
\beq\label{pos1} \mathcal K_{T}(w) =
-\frac{\|\gamma(w)\|^2\|\gamma^\prime(w)\|^2 - |\langle
\gamma^\prime(w), \gamma(w)\rangle|^2}{\|\gamma(w)\|^4}. \eeq
Applying Cauchy - Schwarz inequality, we see that the numerator is positive.

Let $\{e_0,e_1\}$ be an orthonormal set of vectors. Suppose $N$ is a nilpotent
linear transformation defined by the rule
$$e_1 \to a\, e_0,\, e_0  \to 0,\,\, a\in \mathbb C.$$
Then $|a|$ determines the unitary equivalence class of $N$.

The localization $N_T(w)-wI_2= \Big (\begin{smallmatrix} 0 &
h_T(w)\\0 & 0 \end{smallmatrix} \Big )$ of the operator $T$ in
$B_1(\Omega)$ is nilpotent. Now, $h_T(w) > 0$ since we have shown that the
curvature $\mathcal K_T(w)$ is negative. Hence the curvature $\mathcal
K_T(w)$ is an invariant for the operator $T$.
The non-trivial converse of this statement follows from Theorem \ref{eq}.
Thus the operators $T$ and $\wi T$ in $B_1(\Omega)$ are unitarily equivalent if
and only if $N_{T}(w)$ is unitarily equivalent to $ N_{\wi T}(w)$ for $w$ in
$\Omega$.

Note that if $T\in B_1(\mathbb{D})$ is a contraction, that is,
$\|T\|\leq 1$, then $N_{T}(w)$ is a contraction for each
$w\in\mathbb D$. Observe that $\left ( \begin{smallmatrix}
a & c\\
0 & b
\end{smallmatrix} \right )$ is a contraction if and only if
$
|a|\leq 1 \mbox{~and~} |c|^2\leq (1 - |a|^2)(1 - |b|^2).
$
Thus $\|N_{T}(w)\|\leq 1$ if and only if
$\mathcal K_{T}(w)\leq -\frac{1}{(1 - |w|^2)^2},\,\,w\in
\mathbb{D}.$
The adjoint $S^*$ of the unilateral shift  operator $S$ is in $B_1(\mathbb D)$.
It is easy to see that $\g_{S^*}(w) = (1,w,\ldots, w^n,\ldots) \in
\ell^2_+,\,w\in \mathbb D$, is a holomorphic section for the corresponding
holomorphic Hermitian line bundle $E_{S^*}$. The norm $\|\g_{S^*}(w)\|^2$ of the
section $\g_{S^*}$ is $(1-|w|^2)^{-1}$ and hence
the curvature $\mathcal K_{S^*}(w)$ of the operator $S^*$ is given by the
formula
$-\frac{1}{(1-|w|^2)^{2}},\,w\in \mathbb D.$ We have therefore proved:
\begin{prop}
If $T$ is a contractive operator in $B_1(\mathbb D)$, then the curvature of $T$
is bounded above by the curvature of the backward shift operator $S^*$.
\end{prop}

We think of the operator $S^*$ as an extremal operator within the
class of contractions in $B_1(\mathbb D)$. This is a special
case of the curvature inequality
proved in \cite{gm01}.
The curvature inequality is equivalent to contractivity of the operators
$N_T(w)$, $w\in \Omega$, while the contractivity of the operator $T$ is  global
in nature. So, it is natural to expect that the validity of the inequality
$\mathcal K_{T}(w)\leq -\frac{1}{(1 - |w|^2)^2}$, $w\in\mathbb D$,  need not
force
$T$ to be a contraction. Indeed, J. Agler had communicated
the existence of an operator $T$, $\|T\| >1$,  in $B_1(\mathbb D)$  with
$\mathcal K_T(w) \leq \mathcal K_{S^*}(w)$ (cf. \cite[Note added in
proof]{gm01}) to G. Misra. Unfortunately, there is
a printing error in \cite[Note added in proof]{gm01}, which should be corrected
by reversing the inequality sign. No explicit example has been written down of
this phenomenon. We provide such an example here.

The main point of this note is to investigate additional conditions on the
curvature, apart from the inequality we have discussed above, which will ensure
contractivity. We give an alternative proof the curvature inequality.
A stronger inequality becomes apparent from this proof. It is this  stronger
inequality which, as we will show below, admits a converse.

An operator $T$ in the class $B_1(\Omega)$, as is well-known (cf.
\cite[pp. 194 ]{cd}), is unitarily equivalent to the adjoint $M^*$
of the multiplication operator $M$ by the co-ordinate function on
some Hilbert space $\mathcal H_K$ of holomorphic functions on
$\Omega^*:=\{z\in \mathbb C:\bar{z} \in \Omega\}$ possessing a
reproducing kernel $K$.

The kernel $K$ is a complex valued function defined on $\Omega^*\times\Omega^*$
which is holomorphic in the first variable and anti-holomorphic in the second.
In consequence, the map $\bar{w} \to K(\cdot, w)$, $w\in \Omega^*$, is
holomorphic on $\Omega$.
We have $K(z,w) = \overline{K(w,z)}$ making it Hermitian. It is positive
definite in the sense that the $n\times n$ matrix
$$
\big (\!\big (\,K(w_i,w_j)\big )\!\big)_{i,j=1}^n
$$
is positive definite for every subset $\{w_1,\ldots ,w_n\}$ of $\Omega^*$, $n\in
\mathbb N$. Finally, the kernel $K$ reproduces the value of functions
in $\mathcal H_K$, that is, for any fixed $w\in \Omega^*$, the holomorphic
function  $K(\cdot,w)$ belongs to $\mathcal H_K$ and
$$
f(w) = \langle f , K(\cdot, w) \rangle, \,\, f\in \mathcal H_K,\,\, w \in
\Omega^*.
$$

The correspondence between the operator $T$ in $B_1(\Omega)$ and
the operator $M^*$ on the Hilbert space $\mathcal H_K$ is easy to
describe (cf. \cite[pp. 194 ]{cd}). Let $\g$ be a non-zero
holomorphic section (for bounded domain in $\mathbb C$, by Grauert's Theorem,
a global section exists) for the operator $T$ acting on the Hilbert
space $\mathcal H$. Consider the map $\Gamma:\mathcal H \ra
\mathcal O(\Omega^*)$, where $ \mathcal O(\Omega^*)$ is the space
of holomorphic functions on $\Omega^*$, defined by $\Gamma(x)(z) =
\langle x,\g(\bar z) \rangle,\, z\in\Omega^*$. Transplant the
inner product from $\mathcal H$ on the range of $\Gamma$.  The map
$\Gamma$  is now  unitary  from $\mathcal H$ onto the completion
of ${\rm ran}\,\Gamma$. Define $K$ to be the function $K(z,w) =
\Gamma\big (\g(\bar{w})\big )(z) = \langle \g(\bar{w}) ,
\g(\bar{z}) \rangle$, $z,w\in \Omega^*$. Set $K_w(\cdot):=
K(\cdot,w)$. Thus $K_w$ is the function $\Gamma \big
(\g(\bar{w})\big )$. It is then easily verified that $K$ has the
reproducing property, that is,
\begin{eqnarray*}
\langle \Gamma(x)(z), K(z,w) \rangle_{{\rm ran}\,\Gamma} &=&
\langle \big (\langle x, \g(\bar{z}) \rangle \big ) ,\big (\langle \g(\bar{w}) ,
\g(\bar{z}) \big ) \rangle_{{\rm ran}\, \Gamma}\\
&=& \langle \Gamma x, \Gamma(\g(\bar{w})) \rangle_{{\rm ran}\,
\Gamma} = \langle x , \g(\bar{w}) \rangle_\mathcal H \\&=&
\Gamma(x)(w),\,\, x\in \mathcal H,\,\, w\in \Omega^*.
\end{eqnarray*}
It follows that $\|K_w(\cdot)\|^2 = K(w,w)$, $w \in \Omega^*$. Also,
$K_w(\cdot)$ is an eigenvector for the operator $\Gamma \,T\,\Gamma^*$
with eigenvalue $\bar{w}$ in $\Omega$:
\begin{eqnarray*}
\Gamma \,T\, \Gamma^*( K_w(\cdot)) &=& \Gamma\, T\, \Gamma^* \big (
\Gamma(\g(\bar{w}))\big )\\
&=&\Gamma\, T \,\g(\bar{w})\\ &=& \Gamma\, \bar{w}\, \g(\bar{w})\\
&=& \bar{w}\, K_w(\cdot), \,\, w\in \Omega^*.
\end{eqnarray*}
Since the  linear span of the vectors $\{K_w :
w\in\Omega^*\}$ is dense in $\mathcal H_K$, it follows that $\Gamma \,T\,
\Gamma^*$ is the adjoint $M^*$ of the multiplication operator $M$ on $\mathcal
H_K$. We therefore assume, without loss
of generality, that an operator $T$ in $B_1(\Omega)$ has been realized as the
adjoint $M^*$ of the multiplication operator $M$ on some Hilbert space $\mathcal
H_K$ of holomorphic functions on $\Omega^*$ possessing a reproducing kernel
$K$.

\begin{rem}\label{second} The contractivity of the adjoint $M^*$ of the multiplication
operator $M$ on some reproducing kernel Hilbert space $\mathcal
H_K$ is equivalent to the requirement that ${K}^\ddag(z,w):=(1 -
z\bar{w})K(z,w)$ is positive definite on $\mathbb D$ (cf.
\cite[Corollary 2.37]{am01} and \cite[Lemma 1]{dms}). Suppose that
the operator $M^*$  is in $B_1(\mathbb D)$. Here is a second proof
of the curvature inequality:

We have \beqa \frac{~\partial^2}{\partial{w}\partial{\bar{w}}}
\log K(w,w) = \frac{~\partial^2}{\partial{w}\partial{\bar{w}}}
\log {\frac{1}{(1 - |w|^2)}} +
\frac{~\partial^2}{\partial{w}\partial{\bar{w}}}\log K^\ddag(w,w),
\, w\in \mathbb D, \eeqa which we rewrite as \beqa \mathcal
K_{M^*}(w) = \mathcal K_{S^*}(w) -
\frac{~\partial^2}{\partial{w}\partial{\bar{w}}}\log K^\ddag(w,w),
w\in\mathbb D. \eeqa Recalling that
$\frac{~\partial^2}{\partial{w}\partial{\bar{w}}}\log
K^\ddag(w,w)$ must be positive (see \eqref{pos1}) as long as
$K^\ddag$ is positive definite, we conclude that
$$
\mathcal K_{M^*}(w) \leq \mathcal K_{S^*}(w), \,\,w\in \mathbb D.
$$
\end{rem}

The fibre at $\bar w$ of the holomorphic bundle $E_{M^*}$ for $M^*$ in
$B_1(\Omega)$ is the one-dimensional kernel at $\bar w$ of the operator
$M^*$  spanned by $K_w(\cdot)$, $w\in \Omega^*$. In general, there is no obvious
way to define an  inner product between the two vectors
$K_w(\cdot)$ and $(\tfrac{\partial}{\partial\bar{w}} K_w)(\cdot)$.
However since these vectors belong to the same Hilbert space (cf.
\cite[Lemma 4.3]{cs}), in our special case, there is a natural
inner product defined between them. This ensures, via the
Cauchy-Schwarz inequality, the negativity of the curvature
$\mathcal K_T$. The reproducing kernel function $K$ of the Hilbert
space $\mathcal H_K$ encodes the mutual inner products of the
vectors $\{K_w(\cdot): w \in \Omega^*\}$. The Cauchy-Schwarz
inequality, in turn, is just the positivity of the Gramian of the
two vectors $K_w(\cdot)$ and $(\frac{\partial}{\partial\bar{w}}
K_w) (\cdot)$, $w\in\Omega^*$.
The positive definiteness of $K$ is a much stronger positivity requirement
involving all the derivatives of the holomorphic section $K_w(\cdot)$ defined on
$\Omega^*$. We exploit this to show that the curvature function $\big (
\tfrac{~\partial^2}{\partial{z}\partial{\bar{w}}} \log K \big ) (z,w)$ is
actually negative definite not just negative,
whenever $K^t$ is assumed to be positive definite for all $t>0$.

We now construct an example of an operator which is not
contractive but its curvature is dominated by the curvature of the
backward shift. Expanding  the function $K(z,w) = \frac{8+ 8z\bar
w - z^2\bar w^2}{1 - z\bar w}$ in $z\bar{w}$, we see that it has
the form $8 +16 z\bar{w} + 15 \tfrac{z^2\bar{w}^2} {1-z\bar{w}}$.
Therefore, it defines a positive definite kernel on the unit disk
$\mathbb D$. The monomials $\tfrac{z^n}{\|z^n\|}$ with $\|1\|^2
=\tfrac{1}{8}$, $\|z\|^2=\tfrac{1}{16}$ and $\|z^n\|^2 =
\tfrac{1}{15}$ for $n \geq 2$ forms an orthonormal basis in the
corresponding Hilbert space $\mathcal H_K$. The multiplication
operator $M$ maps $\tfrac{z^n}{\|z^n\|}$ to
$\tfrac{\|z^{n+1}\|}{\|z^n\|} \tfrac{z^{n+1}}{\|z^{n+1}\|}$. Hence
it corresponds to a weighted shift operator $W$ with the weight
sequence $\{\sqrt{\tfrac{1}{2}},\sqrt{\tfrac{16}{15}},1,1,\ldots
\}$. Evidently, it is not a contraction. (This is the same as
saying that the function $ K^\ddag(z,w) = 8+ 8z\bar w - z^2\bar
w^2$ is not positive definite.) The operator $W$ is similar to the
forward shift $S$.  Since the class  $B_1(\mathbb D)$ is invariant
under similarity and $S \in B_1(\mathbb D)$, it follows that $W$
is in it as well.  However,
$$
- \frac{~\partial^2}{\partial{w}\partial{\bar{w}}}\log K^\ddag
(w,w) = - \frac{8(8 - 4|w|^2 - |w|^4)}{(8 + 8|w|^2 -
|w|^4)^2},\,\, w\in\mathbb D,
$$
is negative for $|w| < 1$. Hence we have shown that
$\mathcal K_{M^*}(w)= - \frac{~\partial^2}{\partial{w}\partial{\bar{w}}}\log
K(w,w) \leq \mathcal K_{S^*}(w)$, $w\in\mathbb D$, although $M^*$
is not a contraction.

This is not an isolated example, it is easy to modify this example to produce
a family of examples parameterized by a real parameter.

In the following section, we discuss the case of a commuting tuple
${\T}=(T_1,\ldots ,T_m)$ of operators in $B_1(\Omega)$, $\Omega \subseteq
\mathbb C^m$, $m\geq 1$. Even in this case, as before, it is possible to
associate a holomorphic Hermitian bundle $E_{\T}$ to the operator tuple $\T$
such that the equivalence class of the commuting $m$-tuple $\T$ determines the
equivalence class of the bundle $E_{\T}$ and conversely.  We show that the
co-efficient matrix ${\sf K}_T(w)$ of the curvature $(1,1)$ form $\mathcal
K_{\T}$ of the holomorphic Hermitian vector bundle $E_{\T}$ is negative
definite for each $w\in \Omega$. The negativity of the curvature provides an
alternative proof of the curvature inequality given in \cite{dms}.

In the third section, we show that the curvature ${\sf K}_T$ is
negative definite, that is, $\big(\!\big ( {\sf K}_T(w_i,w_j)
\big )\! \big )$ is negative-definite for all finite subsets
$\{w_1,\ldots , w_n\}$ of $\Omega$ if we impose the additional
condition of ``infinite divisibility'' on the reproducing kernel
$K$. The infinite divisibility of the kernel $K$ requires $K^t$ to
be positive-definite for all $t>0$.

In the final section, we give several applications of the positive
definiteness of the curvature function to contractivity of
operators in the Cowen-Douglas class.


\section{Negativity of the curvature in general}\label{2}

Let $\Omega$ be a bounded domain in $\mathbb C^m$. Let ${\T} =
(T_1,\ldots, T_m)$ be a $m$-tuple of commuting operators on a
separable complex Hilbert space $\mathcal H$. For  $x\in \mathcal
H$, let ${D}_{\T}:\mathcal H \ra \mathcal
H\oplus\ldots\oplus\mathcal H$ be the operator defined by
${D}_{\T}(x)=(T_1x,\ldots,T_mx)$. For $w =
(w_1,\ldots,w_m)\in\Omega$, let ${\T}-w$ denote the operator tuple
$(T_1-w_1,\ldots,T_m-w_m)$. The joint kernel of $\T - w$  is
$\cap_{j=1}^m \ker (T_j -w_j)$, which is also the kernel of the
operator ${D}_{{\T}-w}$. Following \cite{cs}, we say that the commuting
tuple $\T$ belongs to the Cowen-Douglas class $B_n(\Omega)$ if
 ${\rm ran}\, D_{\T - w}$ is closed,
 $\dim\,{\ker\, D_{\T-w}} =n$ for all $w$ in $\Omega$, and the
 span of $\{\ker\, D_{\T-w}: w\in\Omega\}$ is dense in $\mathcal H$.
The class of the corresponding holomorphic Hermitian vector bundle
$$E_{\T}= \{(w, x)\in\Omega\times\mathcal H: x\in \ker~
{D}_{{\T}- w}\}$$ determines the class of the operator tuple $\T$.
As before, if $n=1$, then the curvature of $E_{\T}$ (cf.
\cite{cd1,cs}) determines the unitary equivalence class of $\T$.
If $\gamma$ is a non-zero holomorphic section
of the holomorphic Hermitian line bundle $E_{\T}$ defined on some open subset $\Omega_0 \subseteq \Omega$,
then the curvature of the line bundle $E_{\T}$ is the $(1,1)$ form
\beqa \mathcal K_{\T}(w) = -\d_{i,j=1}^{m}\frac{~\partial^2\log
\|{\gamma(w)}\|^2} {\partial{w_i}\partial{\bar w_j}}dw_i\wedge
d\bar w_j,\,w\in\Omega_0, \eeqa defined on
$\Omega_0$.  Let \beq {\sf
K}_{\mathbf{\T}}(w) = \big{(} \! \big (-
\frac{~\partial^2\log\parallel{\gamma(w)}\parallel^2}{\partial{w_i}
\partial{\bar w_j}}\big )\!\big{)} _{i,j=1}^{m},\, w\in\Omega_0,
\eeq
denotes the curvature matrix. In general, for a holomorphic Hermitian vector
bundle, there are two well-known notions of positivity due to Nakano and Griffiths (cf. \cite[page - 338]{dem}).
These two notions coincide in the case of a line bundle, and one talks of
positive line bundle in an unambiguous manner.  The following Proposition shows that
the line bundle corresponding to a commuting tuple of operators in $B_1(\Omega)$
is negative.
\begin{prop} \label{p}
For an operator $\T$ in $\mathrm B_1(\Omega^*)$, the matrix ${\sf
K}_{\mathbf{\T}}(w)$ is negative definite for
each $w\in\Omega^*$.
\end{prop}
\begin{proof}[First Proof]
Fix $w_0\in\Omega$. As before (cf. \cite{cs}), it follows that $\T$ can be realized as ${\M}^* =
(M_1^*,\ldots,M_m^*)$ where $M_i$ is the multiplication operator by the
co-ordinate function $z_i$ on the Hilbert space $\mathcal H_K$ of holomorphic functions on
$\Omega_0\subseteq\Omega,\, w_0\in\Omega_0,$ possessing a reproducing kernel $K$ with $K(w,w)\neq 0$ for $w\in\Omega_0$.
 The function
$$ K_0(z,w) = K(w_0,w_0)^{\frac{1}{2}}\varphi(z)^{-1}K(z,w)\ov{\varphi(w)^{-1}}
K(w_0,w_0)^{\frac{1}{2}}
$$
is defined on some open neighborhood $U\times U$ of $(w_0,w_0)$, where
$U$ is the open set on which $K(z,w_0)$ is non-zero and $\varphi(z) =
K(z,w_0)$ is holomorphic on $U$. The kernel $K_0$ is said to be normalized at
$w_0$(\cite{cs}).
The operator of multiplication by the holomorphic function $\varphi^{-1}$ then
defines a unitary operator from the Hilbert space $\mathcal H_K$ determined
by the kernel function $K$ to the Hilbert space $\mathcal H_{K_0}$ determined by the normalized kernel function $K_0$.
This unitary operator intertwines the two  multiplication operators on $\mathcal H_K$ and $\mathcal H_{K_0}$ respectively.
Thus $\mathcal K_{\M^*}(w_0)$ is equal to the curvature $\mathcal
K_{{\M^{(0)}}^*}(w_0)$ \cite[Lemma 3.9]{cs}, where
${\M^{(0)}}$ is the $m$-tuple of multiplication operator by the co-ordinate
function $z_i$ on the Hilbert space $\mathcal H_{K_0}$. Let
$$
K_0(z,w) = \sum_{I,J}a_{IJ}(z-w_0)^I(\bar w - \bar w_0)^J,\, z,w\in U,\,I,
J\in \mathbb Z_+^m,
$$
be the power series expansion of $K_0$ around the point $(w_0,w_0)$.
Since $K_0(z,w_0) = 1$, we have that
$a_{00} = 1$ and $a_{I0} = 0$ for all $I$ with $|I|>0$. Similarly,
$K_0(w_0,z) = \ov{K_0(z,w_0)}$ shows that $a_{0J} = 0$ for all $J$ with
$|J|>0$. Also note that if
$$
K_0(z,w)^{-1} = \sum_{I,J}b_{IJ}(z-w_0)^I(\bar w- \bar w_0)^J,\,
z,w\in U,\,I, J\in \mathbb Z_+^m,
$$
then $b_{00} = 1$ and $b_{I0} = 0 = b_{0J}$ for all $I,J$ with
$|I|,|J|>0$. Since $\gamma(w) = K_0(\cdot, \bar w), w\in U^*:=
\{\bar z : z\in U\}$ is a section of the holomorphic Hermitian
line bundle $E_{\wt\M^*}$ over $U^*$, we have \beqa
\lefteqn{\frac{~\partial^2\log\parallel{\gamma(w)}\parallel^2}{\partial{w_i}\partial{
\bar w_j}}\big{|}_{w = w_0}} \\ = && \frac{\partial}{\partial\bar
w_j}(K_0(\bar w,\bar w)^{-1}\frac{\partial}
{\partial w_i}K_0(\bar w,\bar w))\big{|}_{w = w_0}\\
= &&\frac{\partial}{\partial\bar w_j} \{ (1 + \sum_{{}_{|I|,|J|\geq
1}}b_{{}_{IJ}}(\bar w - \bar w_0)^I(w-w_0)^J) (\sum_{{}_{|I|\geq 1,|J|\geq
0}}a_{{}_{I{J+\varepsilon_i}}}{\scriptstyle{(J_i+1)}}(\bar w - \bar
w_0)^I(w-w_0)^J)\}\big{|}_{w = w_0}\\&& =
a_{\varepsilon_j\varepsilon_i} \eeqa where $\varepsilon_i$ is the
standard unit vector in $\C^m$ with $1$ at the $i$-th co-ordinate
and $0$ elsewhere. On the other hand, we have
$$
a_{\varepsilon_j\varepsilon_i} = \frac{~\partial^2 K_0(\bar w, \bar
w)}{\partial{w_i}\partial{\bar
w_j}}\big{|}_{w = w_0} = \langle \frac{\partial}{\partial w_i}K_0(\cdot,
\bar w), \frac{\partial}{\partial w_j} K_0(\cdot, \bar w) \rangle\big{|}_{w
= w_0}.
$$
Thus for any complex constants $\a_1,\ldots,\a_m$,
$$
-\sum_{i,j=1}^m \a_i\bar\a_j
\frac{~\partial^2\log\parallel{\gamma(w)}\parallel^2}{\partial{w_i}\partial{\bar
w_j}}\big{|}_{w = w_0} = -\|\sum_{i=1}^m\a_i
\frac{\partial}{\partial w_i}K_0(\cdot, \bar w)\|^2\big{|}_{w
= w_0}\leq 0.
$$
This completes the proof.

\noindent\textit{Second Proof}. We show that $-{\sf
K}_{\mathbf{\T}}(w)$ is the Gramian of a set of $n$ vectors which is
explicitly exhibited below. These vectors are
$$
e_i(w) = K_w\otimes \frac{\partial}{\partial\bar{w_i}} K_w -
\frac{\partial}{\partial\bar{w_i}} K_w \otimes K_w, \,\, 1\leq i
\leq n,
$$
in $\mathcal H_K\otimes \mathcal H_K$.
Then \beqa \langle e_i(w), e_j(w)\rangle &=& \langle  K_w\otimes
\frac{\partial}{\partial\bar{w_i}} K_w -
\frac{\partial}{\partial\bar{w_i}} K_w \otimes K_w,
 K_w\otimes \frac{\partial}{\partial\bar{w_j}} K_w -
\frac{\partial}{\partial\bar{w_j}} K_w \otimes K_w\rangle \\ &=&
2(K(w,w)\frac{\partial^2K(w,w)}{\partial w_i\partial \bar w_j} -
\frac{\partial}{\partial w_i}K(w,w)\frac{\partial}{\partial\bar
w_j}K(w,w)). \eeqa Thus \beqa
\frac{~\partial^2\log\parallel{\gamma(w)}\parallel^2}{\partial{w_i}\partial{
\bar w_j}}\big{|}_{w = w_0} &=&
\frac{K(w,w)\frac{\partial^2K(w,w)}{\partial w_i\partial \bar w_j}
- \frac{\partial}{\partial w_i}K(w,w)\frac{\partial}{\partial\bar
w_j}K(w,w)}{K(w,w)^2}\big{|}_{w = w_0}\\ &=& \frac{\langle e_i(w_0),
e_j(w_0)\rangle}{2K(w_0,w_0)^2}. \eeqa This completes the proof.
\end{proof}

A commuting tuple of operators $\T = (T_1, \ldots, T_m)$ is
said to be a row contraction if $\sum_{i=1}^m T_iT_i^*\leq I.$
The following characterization of row contractions is well known
(cf. \cite[Corollary 2]{dms}).

\begin{lem}\label{c4lem}
Let $\mathbb{B}^m$ be the unit ball in $\mathbb{C}^m$ and
${\M}=(M_1,\ldots,M_m)$ be $m$-tuples of
multiplication operator on reproducing kernel Hilbert space with
reproducing kernel $K$. Then $\mathbf{\M}$ is a row contraction if and only if $(1-\langle
z,w\rangle) K(z,w)$ is positive definite.
\end{lem}

Let ${\boldsymbol R}_m^*$ be the adjoint of the joint weighted shift operator on the
Drury-Arveson space $H^2_m$. This is the commuting
tuple $(M_1^*,\ldots,M_m^*)$ on $H^2_m$ which is determined by the
reproducing kernel $\tfrac{1}{1- \langle z,  w\rangle},\,
z = (z_1,\ldots, z_m) ,\, w = (w_1,\ldots, w_m) \in \mathbb B^m$. As in Remark \ref{second}, using
Proposition \ref{p} and  Lemma \ref{c4lem}, we obtain a version of curvature inequality for the
multi-variate case. It appeared earlier in \cite{dms} with a different proof.

\begin{cor}
If $\T = (T_1, \ldots, T_m)$ is a row contraction in $
B_1(\mathbb{B}^m)$, then ${\sf K}_{{\boldsymbol R}_m^*}(w) - {\sf
K}_{\T}(w)$ is positive for each $w$ in the unit ball
$\mathbb B^m$.
\end{cor}

\section{Infinite divisibility and curvature inequality}\label{3}
Starting with a  positive definite kernel  $K$  on a bounded
domain $\Omega$ in $\mathbb C^m$, it is possible to construct
several new positive definite kernel functions. For instance, if
$K$ is positive definite then the kernel $K^n,\, n\in {\mathbb
N},$ is also positive definite. Indeed, a positive definite kernel
$K$ is said to be \emph{infinitely divisible} if for all $t > 0$,
the kernel $K^t$ is also positive definite. While the Bergman
kernel for the Euclidean ball is easily seen to be infinitely
divisible, it is not infinitely divisible for the unit ball (with
respect to the operator norm) of the $n\times n$ matrices. We give
the details for $n=2$  in the final Section of this note. The
following Lemma shows that if $K$ is positive definite then the
matrix valued kernel $\big(\!\big (\tfrac{\partial^2}{\partial
z_i\,\partial \bar{w_j}} K(z,w)\, \big )\!\big )_{i,j=1}^m$ is positive
definite as well.
\begin{lem}\label{dpd}
For any bounded domain $\Omega$ in $\C^m$, if $K$ defines a
positive definite kernel  on $\Omega$, then $-{\sf K}(z,w)=\big(\!\big
(\tfrac{\partial^2}{\partial z_i\, \partial \bar{w_j}} K(z,w)\,\big )\!\big
)_{i,j=1}^m$ is
a positive definite kernel on $\Omega$.
\end{lem}
\begin{proof}
Let  $\xi_i = (\xi_i(1),\ldots,\xi_i(m))$, $1\leq i\leq m$, be vectors in
$\C^m$ and $u_1,\ldots,u_n$ be an arbitrary set of $n$ points in $\Omega$.
Since $\bar\partial_iK_w$ belongs to $\mathcal H_K$, as
shown in \cite{cs}, it follows that \beqa \sum_{i,j}^n\langle {-\sf
K}(u_i,u_j)\xi_j,\xi_i\rangle_{\C^m} &=&  \sum_{i,j}^n
\sum_{k,l}^m \big (\tfrac{\partial^2}{\partial w_k\,
\partial \bar{w_l}} K \big )(u_i,u_j)\xi_j(l)\ov{\xi_i(k)}\\ &=& \sum_{i,j}^n \sum_{k,l}^m \langle \tfrac{\partial}{
\partial \bar{w_l}}K_{u_j}, \tfrac{\partial}{
\partial \bar{w_k}}K_{u_i}\rangle_{\mathcal H_K}\xi_j(l)\ov{\xi_i(k)}\\&=&
\|\sum_{i}^n \sum_{k}^m\xi_i(k)\tfrac{\partial}{
\partial \bar{w_k}}K_{u_i}\|_{\mathcal H_K}^2 \\&\geq & 0
\eeqa
This completes the proof.
\end{proof}

\begin{rem}
Even in the case of one variable, the proof of the Lemma given is interesting. In fact, this motivates
the proof of the main theorem (Theorem \ref{main}) in one direction. In
particular, it says that if $K$ is a positive definite kernel on a bounded
domain $\Omega\subset\C$, then
$\big ( \tfrac{\partial^2}{\partial z\, \partial \bar{w}} K \big )(z,w)$ is also
a positive definite kernel on $\Omega$.
Without loss of generality, assume that $0$ is in  $\Omega$  and let $K(z,w) =
\sum_{m,n}^\infty a_{mn}z^m\bar w^n$ be the power series expansion of $K$ around
$0$. It is shown in \cite[Lemma 4.1 and 4.3]{cs} that the positivity of the kernel
$K$ is equivalent to the positivity of the matrix of Taylor co-efficients of $K$
at $0$, namely,
\beqa
H_{n}(0 ; K):= \begin{pmatrix}
a_{00} & a_{01} & a_{02} & \cdots & a_{0n} \\
a_{10} & a_{11} & a_{12} & \cdots & a_{1n} \\
\vdots & \vdots & \vdots & \ddots & \vdots \\
a_{n0} & a_{n1} & a_{n2} & \cdots & a_{nn} \\
\end{pmatrix}
\eeqa
for each $n\in \mathbb Z_+$.
The function $\tfrac{\partial^2}{\partial z\,\partial\bar{w}} K(z,w)$ admits the expansion
$$\sum_{m,n=0}^\infty {(m+1)(n+1)}a_{(m+1)(n+1)}z^m\bar w^n.$$ Therefore, for $n\in\N$,

$$H_{n - 1}(0 ; \frac{\partial^2}{\partial z\,\partial\bar{w}} K)=\begin{pmatrix}
a_{11} & 2a_{12} & \cdots & na_{1n} \\
2a_{21} & 4a_{22} & \cdots & 2na_{2n} \\
\vdots &  \vdots & \ddots & \vdots \\
na_{n1} & 2na_{n2} & \cdots & n^2a_{nn} \\
\end{pmatrix}.
$$
Clearly, for $n\in\N$, we have
$$
\Big (\begin{smallmatrix}0_{1\times 1}&0_{1\times n}\\ 0_{n\times 1}&H_{n -1}(0 ;
 \frac{\partial^2}{\partial z\,\partial\bar{w}} K)\end{smallmatrix}\Big )=  D\big (H_{n}(0 ; K) \big ) D ,
$$
where $D:\C^{n+1} \to \C^{n+1}$ is the linear map which is
diagonal and is  determined by the sequence $\{0,1,\ldots,
k,\ldots, n \}$. It therefore follows that $H_{n}(0 ;
\frac{\partial^2}{\partial z\,\partial\bar{w}} K)$ is positive
definite for all $n\in \mathbb Z_+$. Consequently,
$\tfrac{\partial^2}{\partial z\,\partial\bar{w}}K$ is a positive
definite kernel.
\end{rem}

The following Lemma encodes a way to extract scalar valued positive definite kernel from the matrix valued one.

\begin{lem}\label{inn}
If $K$ is a $n\times n$ matrix valued positive definite kernel on a bounded domain $\Omega\subset\C^m$, then for every $\zeta\in\C^n$,
$\langle K(z, w)\zeta,\zeta\rangle_{\C^n}$ is also
a positive definite kernel on $\Omega$.
\end{lem}
\begin{proof}
Let $K_{\zeta}(z,w) = \langle K(z, w)\zeta,\zeta\rangle_{\C^n}$.
Let $u_1,\ldots,u_l$ be $l$ points in $\Omega$ and $\a_i$, $1\leq
i\leq l$, be scalars in $\C$. From \cite{cs}, it follows that
\beqa
\sum_{i,j}^l\a_iK_{\zeta}(u_i, u_j)\bar\a_j &=& \sum_{i,j}^l \a_i\bar\a_j\langle K(\cdot, u_j)\zeta, K(\cdot, u_i)\zeta\rangle_{\mathcal H_K}\\
&=& \|\sum_{j}^l\bar\a_j K(\cdot, u_j)\zeta\|_{\mathcal H_K}^2\\&\geq & 0
\eeqa
This completes the proof.
\end{proof}

\begin{defn}
Let $G$ be a real analytic function of $w,\bar{w}$ for $w$ in some
open connected subset $\Omega$ of $\mathbb C^n$. Polarizing $G$,
we obtain a (unique) new function $\wt G$ defined on $\Omega\times \Omega$
which is holomorphic in the first variable and anti-holomorphic in
the second and restricts to $G$ on the diagonal set $\{(w,w): w\in
\Omega\}$, that is, $\tilde{G}(w,w)=G(w,w)$,  $w\in\Omega$. If the
function $\tilde{G}$ is also positive definite, that is, the
$n\times n$ matrix $ \big (\!\big ( \wt{G}(w_i,\bar{w}_j)
\big)\!\big ) $ is positive definite for all finite subsets
$\{w_1,\ldots , w_n\}$ of $\Omega$, then we say that $G$ is a
positive definite function on $\Omega$.
\end{defn}
The curvature $\mathcal K$ of a line bundle is a real analytic
function. We have shown that $-{\sf K}(w),$
$w\in\Omega\subset\mathbb{C}^m,$ is positive definite. However,
the following example shows that $-{\sf K}(w)$ need not be a
positive definite function, that is, $-\widetilde{{\sf K}}(w)$
need not be positive definite! We adopt the convention that the
positive definiteness of the real analytic function $-{\sf K}(w)$
is the same as the positive definiteness of the Hermitian function
$-\widetilde{{\sf K}}(w)$.

\begin{ex} Let $K(z,w) = 1+ \sum_{i=1}^\infty a_iz^i\bar w^i$ be a
positive definite kernel on the unit disc $\mathbb D$. The kernel $K$ then
admits a power series expansion some small neighborhood of $0$.  Consequently,
we
have
\beqa
\log K (z,w) &=& \log (1+ \sum_{i=1}^\infty a_iz^i\bar w^i)\\
&=& \sum_{i=1}^\infty a_iz^i\bar w^i - \frac{(\sum_{i=1}^\infty a_iz^i\bar w^i)^2}{2} + \frac{(\sum_{i=1}^\infty a_iz^i\bar w^i)^3}{3} - \cdots \\
&=& a_1z\bar w + (a_2 - \tfrac{a_1^2}{2})z^2\bar w^2 + (a_3 - a_1a_2 +
\tfrac{a_1^3}{3})z^3\bar w^3 + \ldots
\eeqa
It follows that
\beqa
\big (\tfrac{\partial^2}{\partial z\,\partial\bar{w}} \log  K\big )(z,w) ~=~ a_1 + 4 (a_2 - \tfrac{a_1^2}{2})z\bar w + 9 (a_3 - a_1a_2 + \tfrac{a_1^3}{3})z^2\bar w^2 + \ldots
\eeqa
Thus if we choose $0<a_i$, $i\in \mathbb N$, such that $a_2< \frac{a_1^2}{2}$, then from \cite[Lemma 4.1 and 4.3]{cs}, it follows that $\frac{\partial^2}{\partial z\,\partial\bar{w}} \log K$ is
not positive definite.

However we note, for instance, that if $K$ is the function $1 + z\bar w +\frac{1}{4}z^2\bar w^2 +\sum_{i=3}^\infty z^i\bar w^i$, then
$$
K^t(z,w) = 1 + tz\bar w + \frac{t(2t - 1)}{4}z^2\bar w^2+\cdots
$$
is not positive definite for $t<\frac{1}{2}$.
\end{ex}
It is therefore natural to ask if assuming that $K$
is infinitely divisible is both necessary and sufficient for positive
definiteness of the curvature function $-{\sf K}$.
The following Theorem provides an affirmative answer.

\begin{thm}\label{main}
Let $\Omega$ be a domain in $\C^m$ and $K$ be a positive real
analytic function on $\Omega\times\Omega$. If $K$ is infinitely divisible then
there exist a domain $\Omega_0\subseteq \Omega$ such that the
curvature matrix $\big(\!\big (\tfrac{\partial^2}{\partial w_i\,
\partial \bar{w_j}} \log K \big )\!\big )_{i,j=1}^m$ is positive definite function on $\Omega_0$.
Conversely, if the function $\big(\!\big
(\tfrac{\partial^2}{\partial w_i\,
\partial \bar{w_j}}\log \hat{K} \big )\!
\big )_{i,j=1}^m$ is a positive definite on $\Omega$, then there
exist a neighborhood $\Omega_0 \subseteq \Omega$ of $w_0$ and a
infinitely divisible kernel $K$ on $\Omega_0\times \Omega_0$ such
that $K(w,w) = \hat K(w,w)$, for all $w\in\Omega_0$.

\end{thm}
\begin{proof}
For each $t>0$, assume that $K^t$ is positive definite on
$\Omega$. This is the same as  the positive definiteness of  $\exp
t\log K$, $t>0$. Clearly $t^{-1}(\exp t\log K - 1)$ is
conditionally positive definite (An Hermitian kernel $L$ is said
to be conditionally positive definite if for every $n\in\N$ and
for every choice $n$ points $w_1,\ldots, w_n$ and complex scalars
$\a_1,\ldots, \a_n$ with $\sum_{i=1}^n\a_i = 0$, the inequality
$\sum_{i,j = 1}^n\a_i \bar \a_j L(w_i, w_j) \geq 0$ holds). By
letting $t$ tend to $0$, it follows that $\log K$ is
conditionally positive definite. Hence at an arbitrary point in
$\Omega$, in particular at $w_0$, the kernel \beqa L_{w_0}(z,w) =
\log K(z,w) - \log K(z,w_0) - \log K(w_0,w) + \log K(w_0,w_0)
\eeqa is positive definite. This is essentially the Lemma $1.7$ in
\cite{ps}. From Lemma \ref{dpd}, it follows that the matrix
$\big(\!\big (\tfrac{\partial^2}{\partial w_i\,
\partial \bar{w_j}} L_{w_0} \big )\!
\big )$ is positive definite on $\Omega$.
 Note that there exist a neighborhood
  $\Omega_0 \subseteq \Omega$ of $w_0$ such that
  $\log K(z,w_0)$ is holomorphic on $\Omega_0$.
  Hence from the equation above, the curvature matrix
$\big(\!\big (\tfrac{\partial^2}{\partial w_i\,
\partial \bar{w_j}} \log K \big )\!
\big )$ is positive definite on $\Omega_0$.
This proves the Theorem in the forward direction.

For the other direction, without loss of generality,
 assume that $w_0 = 0$. Let ${\sf K}(z,w)$
  be the function obtained by polarizing the
real analytic $m\times m$ matrix valued function
$$\big (\!\big (\tfrac{\partial^2}
{\partial w_i \partial \bar w_j}\log \hat{K}(w,w)\big )\! \big
)_{i,j=1}^m$$ defined on some bounded domain $\Omega$ in $\mathbb
C^m$. Suppose that $\log \hat{K}$ has the power series expansion
$\sum a_{IJ} z^I\bar{w}^J$,  where the sum is over all
multi-indices $I,J$ of length $m$ and $z^I = z_1^{i_1} \cdots
z_m^{i_m}$, $\bar{w}^J = \bar{w}_1^{j_1} \cdots \bar{w}_m^{j_m}$.
Then
$${\sf K}(z,w) = \sum_{I,J} a_{IJ}
\big (\! \big ( A_{IJ}(k,\ell) z^{I-\epsilon_k}
\bar{w}^{J-\epsilon_\ell}\big )\! \big ) ,
$$ where $A_{IJ}(k,\ell) =\big(\!\!\big (\scriptstyle{i_k \,j_\ell}\big)\!\!\big )_{k,\ell=1}^m$
and the sum is again over all multi-indices $I,J$ of size $m$.
Clearly, $A_{IJ}$ can be written as the product $D(I)\,E_m \,
D(J)$, where $D(I)$ and $D(J)$ be the $m\times m$ diagonal
matrices with $(i_1,\ldots ,i_m)$ and $(j_1,\ldots ,j_m)$ on their
diagonal respectively, and $E_m$ is the $m\times m$ matrix all of whose
entries are $1$.

Let $D(z)$ be the holomorphic function on $\Omega$ taking values
in the $m\times m$ diagonal matrices which has $z_i$ in the
$(i,i)$ position for $z :=(z_1,z_2,\ldots, z_m)\in \Omega$. If the
function ${\sf K}$ is assumed to be positive definite then
$$\widetilde{{\sf K}}(z,w):=D(z)\,{\sf K}(z,w) \,D(\bar{w})=
\sum_{I,J} a_{IJ} D(I)\,E_m \, D(J) z^I\bar{w}^J
$$
is positive definite on $\Omega_0$.

Let $\Lambda(I) = \{k: 1\leq k\leq m \mbox{~and~} i_k\neq 0\}$.
Consider the $m\times m$ matrix $E(I,J)$ defined below:
\beqa
E(I,J)_{ij} =  \begin{cases} 1 &\mbox{if}
\,\, i\in\Lambda(I) \mbox{~and~} j\in\Lambda(J),\\ 0  & \mbox{otherwise}. \end{cases}
\eeqa
Note that if $\Lambda(I) = \Lambda(J) = \{1,\ldots,m\}$, then $E(I,J) = E_m$.
Consider the function on $\Omega_0\times \Omega_0$, defined by
$$
\widehat{\widetilde{{\sf K}}}(z,w) = \sum_{I,J\neq 0} a_{IJ}
\tfrac{E(I,J)}{c(I)c(J)} z^I\bar{w}^J,
$$ where $c(I)$ denotes the cardinality of the set $\Lambda(I)$. We will prove that
$\widehat{\widetilde{{\sf K}}}$ is a positive definite kernel on
$\Omega_0$. To facilitate the proof, we need to fix some
notations.

Let $\delta$ be a multi-index of size $m$. Also let $p(\delta) =
\prod_{j=1}^m(\delta_j+1)$ which is the number of multi-indices $I\leq\delta$, that is,
$i_l\leq\delta_l,\, 1\leq l\leq m$. As par the notation in \cite{cs}, given a function $L$ on a domain $U\times U$
which is holomorphic in the first variable and antiholomorphic in the second,
let $H_{\delta} (w_0; L)$ be the $p(\delta)\times p(\delta)$ matrix whose
$(I, J)$-entry is $\tfrac{\partial^I {\bar\partial}^J L(w_0,w_0)}{I!J!},\, 0\leq I,J\leq\delta$.
For convenience, one uses the colexicographic order to write down the matrix, that is,
$I\leq_c J$ if and only if $(i_m<j_m)$ or
$(i_m = j_m \mbox{~and~} i_{m-1}<j_{m-1})$ or
$\cdots$ or $(i_m = j_m \mbox{~and~} \ldots i_2 = j_2 \mbox{~and~} i_1<j_1)$ or $I = J$.

Let $D(I)^\sharp$ be the diagonal matrix with the diagonal entry
$D(I)^\sharp_{\ell\,\ell}$ equal to $\tfrac{1}{i_\ell}$ or $0$
according as $i_\ell$ is non-zero or zero.
Using this notation, we have \beqa D(I)^\sharp D(I)\,E_m \,
D(J)D(J)^\sharp = E(I,J). \eeqa Let $A_\delta$ be the block diagonal
matrix,  written in the colexicographic ordering, of the form
\beqa
(A_\delta)_{IJ} =  \begin{cases} \tfrac{D(I)^\sharp}{c(I)}
&\mbox{if} \,\, I = J (\neq 0)\\ 0  & \mbox{otherwise}.
\end{cases}
\eeqa
Therefore, in this setup, for any multi-index $\delta$, we have
$$
H_{\delta} (0; \widehat{\widetilde{{\sf K}}}) = A_\delta
H_{\delta} (0; {\widetilde{{\sf K}}}) A_\delta^*.
$$
Clearly $H_{\delta} (0; \widehat{\widetilde{{\sf K}}})$ is
positive definite since $H_{\delta} (0; {\widetilde{{\sf K}}})$ is
so by \cite[Lemma 4.1]{cs}. Thus from \cite[Lemma 4.3]{cs}, it
follows that $\widehat{\widetilde{{\sf K}}}$ is a positive
definite kernel.

Let $ K_0$  be the scalar function on $\Omega_0\times\Omega_0$
defined by
$$K_0(z,w):= \sum a_{IJ} z^I \bar{w}^J,$$
where the sum is over all pairs $(I,J)$ excluding those of the
form $(I,0)$ and $(0,J)$. From Lemma \ref{inn}, it follows that
the function  $ K_0$ is positive definite since it is of the form
$\langle \widehat{\widetilde{{\sf K}}}(z,w)\mathbf 1, \mathbf
1\rangle, $ $\mathbf 1 = (1,\ldots, 1)$. It is evident that
$$
\Big (\!\!\Big( \, \big (\tfrac{\partial^2}{\partial w_i \partial
{\bar w}_j} {K_0} \big )(w,w)\, \Big )\!\!\Big )= {\sf K}(w,w),
$$
that is,
$$
\tfrac{\partial^2}{\partial w_i\,\partial\bar{w}_j}(\log \hat K -
K_0(w,w)) = 0,\,\, 1\leq i,j\leq m,\,\, w\in \Omega_0.
$$

Therefore, $\log \hat K - K_0$ is a real pluriharmonic function on
$\Omega_0$ and hence there exist a holomorphic function $\phi$
such that
$$
\log \hat K(w,w) - K_0(w,w) = (\Re \phi)(w) := \frac{\phi(w) +
\ov{\phi(w)}}{2}.
$$

(Alternatively, since $\log \hat K$ is real analytic, it follows that \beqa \sum_{I,J}
a_{IJ}w^I\bar w^J~=~\sum_{I,J}
\bar a_{IJ}w^J\bar w^I\eeqa Equating coefficients, we get $a_{IJ}~=~\bar a_{JI}$
for all multi-indices $I,J$. In particular, we have $a_{I0}~=~\ov{a_{0I}}$ for all multi-indices $I$. The power series
$$(a_{00}/2)+\sum_I a_{I0}{z}^I$$ defines a holomorphic function
$\psi$ on $\Omega_0$ such that $\log \hat K(w,w) - K_0(w,w) =
\psi(w) + \ov{\psi(w)}$.)

\noindent Thus \beqa\label{nice} \hat K(w,w) =
\exp(\tfrac{\phi(w)}{2})\exp(K_0(w,w))\exp(\tfrac{\ov{\phi(w)}}{2}),\,\,w\in
\Omega_0. \eeqa Let $K:\Omega_0\times \Omega_0\ra\C$ be the
function defined by the rule
$$
K(z,w) =
\exp(\tfrac{\phi(z)}{2})\exp(K_0(z,w))\exp(\ov{\tfrac{\phi(w)}{2}}).
$$
For $t>0$, we then have
$$
K^t(z,w) =
\exp(t\tfrac{\phi(z)}{2})\exp(tK_0(z,w))\exp(t\ov{\tfrac{\phi(w)}{2}}),\,\,
z,w\in \Omega_0.
$$
By construction $K(w,w) = \hat K(w,w)$, $w\in \Omega_0$. Since
$K_0$ is a positive definite kernel as shown above, it follows
from \cite[Lemma 1.6]{ps} that $\exp(tK_0)$ is a positive definite
kernel for all $t>0 $ and therefore $K^t$ is a positive definite
on $\Omega_0$ for all $t>0$ completing the proof of the converse.
\end{proof}

\section{Applications}

Let $\M^*$ be the adjoint of the commuting tuple of
multiplication operators acting on the Hilbert space $\mathcal H_K
\subseteq \mathcal O(\Omega)$. Fix a positive definite kernel
$\mathfrak K$ on $\Omega$.  Let us say that $\M$ is infinitely
divisible with respect to $\mathfrak K$ if  $\mathfrak
K(z,w)^{-1}K(z,w)$ is infinitely divisible in some open subset
$\Omega_0$ of $\Omega$. As an immediate application of Theorem \ref{main}, we obtain :

\begin{thm}
Assume that the adjoint $\M^*$ of the multiplication
operator $\M$ on the reproducing kernel Hilbert space
$\mathcal H_K$ belongs to $B_1(\Omega)$. The  function
$\big (\!\!\big (\frac{\partial^2}{\partial w_i\,\partial\bar{w}_j} \log \big
(\mathfrak K(w,w)^{-1} K(w,w) \big )\, \big)\!\!\big )$ is positive definite, if
and only if the multiplication operator $\M$ is infinitely
divisible with respect to $\mathfrak K$.
\end{thm}

If $K$ is a positive definite kernel on $\mathbb D$ such that
$(1-z\bar w)K(z,w)$ is infinitely divisible then we say that $M^*$
on $\mathcal{H}_K$ is a infinitely divisible contraction.

Here is an example showing that a contraction need not be
infinitely divisible. Take
 \begin{eqnarray*}K(z,w)&=&(1-z\bar w)^{-1}\big(1+z\bar
w+\tfrac{1}{4}z^2\bar{w}^2+\sum_{n=3}^{\infty}z^n\bar{w}^n\big)\\
&=& 1+2z\bar w+\sum_{n=2}^{\infty}(n+\tfrac{1}{4})z^n\bar{w}^n.
\end{eqnarray*}
Clearly $K$ defines the positive definite kernel on $\mathbb{D}$.
Since $(1-z\bar w)K(z,w)$ is also positive definite, it follows
that the adjoint of multiplication operator $M^*$ on
$\mathcal{H}_K$ is contractive. But
$$
\big((1-z\bar w)K(z,w)\big)^t = 1 + tz\bar w + \frac{t(2t -
1)}{4}z^2\bar w^2+\cdots
$$
is not positive definite for $t<\frac{1}{2}$ as was pointed out
earlier. Hence $M^*$ is not infinitely divisible contraction on
$\mathcal{H}_K$.

The following Corollary is a characterization of infinitely
divisible contractions in the Cowen-Douglas class $B_1(\mathbb D)$
completing the study begun in \cite{gm01}. Here, for two real
analytic functions $ G_1$ and $ G_2$ on a domain
$\Omega\subset\mathbb{C}^m$, $ G_1(w) \preceq
G_2(w),\,\,w\in\Omega,$ means $ G_2 -  G_1$ is a positive definite
function on $\Omega$.

\begin{cor}\label{mainc1}
Let $K$ be a positive definite kernel on the open unit disc $\mathbb D$.
Assume that the adjoint $M^*$ of the multiplication operator
$M$ on the reproducing kernel Hilbert space $\mathcal H_K$
belongs to $B_1(\mathbb D)$. The  function
$\frac{\partial^2}{\partial w\,\partial\bar{w}} \log \big ((1 -|
w|^2)K(w,w) \big )$ is positive definite, or equivalently
$$ \mathcal K_T(w) \preceq \mathcal K_{S^*}(w),\,\,w\in\mathbb D,$$
if and only if the
multiplication operator $M$ is an infinitely divisible contraction.
\end{cor}
\begin{proof}
Recall Theorem \ref{main}, which says that the positive
definiteness of $$\frac{\partial^2}{\partial w\,\partial\bar{w}}
\log \big ((1 - |w|^2)K(w,w) \big )$$ is equivalent to infinite
divisibility of the kernel $(1 - z\bar w)K(z,w)$, that is, $\big
((1 - z\bar w)K(z,w)\big )^t$ is positive definite for all $t\geq
0$.
\end{proof}


We say that a commuting tuple of multiplication operators $\M$ is
an infinitely divisible row contraction if $(1-\langle z,w\rangle) K(z,w)$
is infinitely divisible, that is, $\big ((1-\langle z,w\rangle) K(z,w) \big )^t$
is positive definite for all $t>0$.

Recall that ${\boldsymbol R}_m^*$ is the adjoint of the joint weighted shift operator on the
Drury-Arveson space $H^2_m$. The
following theorem is a characterization of infinitely divisible
row contractions.

\begin{cor}\label{mainc2}
Let $K$ be a positive definite kernel on the Euclidean ball
$\mathbb B^m$. Assume that the adjoint $\M^*$ of the
multiplication operator $\M$ on the reproducing kernel Hilbert
space $\mathcal H_K$ belongs to $B_1(\mathbb B^m)$. The  function
$\big (\!\big (\frac{\partial^2}{\partial w_i\,\partial\bar{w_j}}
\log \big ((1 -\langle  w,  w \rangle)K( w, w) \big )\,\big
)\!\big )_{i,j=1}^m$, $ w\in \mathbb B^m,$ is positive definite,
or equivalently
$$ {\sf K}_{\M^*}( w) \preceq {\sf K}_{\boldsymbol R_m^*}( w),\,\, w\in\mathbb B^m,$$
if and only if the multiplication operator $\M$ is an infinitely
divisible row contraction.
\end{cor}

We give one last example, namely that of the polydisc $\mathbb
D^m$.  In this case, we say a commuting tuple $\M$ of multiplication
by the co-ordinate functions acting on
the Hilbert space $\mathcal H_K$ is infinitely divisible if $\big
(S( z,w) K( z, w) \big )^t$, where $S( z, w):= \prod_{i=1}^m
(1-z_i\bar{w}_i)^{-1},\,  z, w\in \mathbb D^m$, is positive
definite for all $t>0$.  Every commuting tuple of contractions
$\M^*$ need not be infinitely divisible. Let $\boldsymbol S_m$ be the
commuting $m$ - tuple of the joint weighted shift defined on the
Hardy space $H^2(\mathbb D^m$).

\begin{cor}\label{mainc3}
Let $K$ be a positive definite kernel on the polydisc $\mathbb
D^m$. Assume that the adjoint $\M^*$ of the multiplication
operator $\M$ on the reproducing kernel Hilbert space
$\mathcal H_K$ belongs to $B_1(\mathbb D^m)$. The  function $\big
(\!\big (\frac{\partial^2}{\partial w_i\,\partial\bar{w_j}} \log
\big (S(z, w) K(w,w) \big )\,\big )\!\big )_{i,j=1}^m$, $w\in
\mathbb D^m,$ is positive definite, or equivalently
$$ {\sf K}_{\M^*}(w) \preceq {\sf K}_{\boldsymbol S_m^*}(w),\,\, w\in\mathbb D^m,$$
if and only if the multiplication operator $\M$ is an infinitely
divisible $m$-tuple of contractions.
\end{cor}
 For a second application of these ideas, assume that $K$ is a
positive definite kernel on $\mathbb{D}^m$ with the property:
$$K_i(z,w)= (1-z_i\bar{w}_i)^m K(z,w),\, 1\leq i \leq m,$$ is
infinitely divisible. Then
$$K^m(z,w)=\big(\prod_{i=1}^{m}(1-z_i\bar{w}_i)\big)^{-m}\prod_{i=1}^{m}K_i(z,w).$$
It now follows that
$$K(z,w)= S(z,w)\big(\prod_{i=1}^{m}K_i(z,w)\big)^{\frac{1}{m}}.$$
Let $\M$ be the commuting tuple of multiplication operators on the Hilbert space
$\mathcal H_K$, which is contractive since $K$ admits the S\"{z}ego kernel $S$ as a factor.
Clearly, the infinite divisibility of $K_i$, $1\leq i \leq m$, implies that
$\big(\prod_{i=1}^{m}K_i(z,w)\big)^{\frac{1}{m}}$ is positive definite.
As pointed out in \cite{dms}, in consequence, for any polynomial $p$ in $m$ - variables,
$$
 p(M_1, \ldots , M_m)= P_{\mathcal S}p(S_m)_{| \mathcal S},
$$
where $\mathcal S$ is the invariant subspace of functions vanishing on the diagonal
of the Hilbert space $H^2_m \otimes \mathcal H_K\subseteq \mathcal O(\mathbb D^m \times \mathbb D^m)$
and $P_\mathcal S$ is the projection onto the subspace $\mathcal S$.

A basic question raised in the paper of Cowen and Douglas
\cite[Section 4]{cd} is the determination of non-degenerate
holomorphic curves in the Grassmannian. Clearly, a necessary
condition for this is the positive definiteness of the curvature
function.  Thus we have the following corollary to Theorem
\ref{main}.
\begin{cor}
Let $E$ be a holomorphic Hermitian  vector bundle of rank $1$ over a
bounded domain $\Omega \subset \mathbb C^m$. If the curvature
${\sf K}$ of $E$ is negative definite, then there exists a
Hilbert space $\mathcal{H}$ and a holomorphic map $\gamma:\Omega_0
\to\mathcal{H}$, $\Omega_0$ open in $\Omega$, such that $E$  is
isomorphic to $E_{\gamma}$, where $E_{\gamma}$ is the pullback, by
the holomorphic map
$\gamma:\Omega_0\to\mathcal{G}r(1,\mathcal{H})$, of the
tautological bundle defined over $\mathcal{G}r(1,\mathcal{H})$.
Moreover, the real analytic function $\langle \gamma (z),\gamma
(w)\rangle$ defined on $\Omega_0\times\Omega_0$ is infinitely
divisible.
\end{cor}


We finish with an amusing application of the Lemma 1.7 in \cite{ps} which is
a key ingredient in the proof of Theorem \ref{main}. Let $K$ be the function on the unit ball
$\mathbb B_{2\times 2}$ (with respect to the operator norm) of $2\times 2$ matrices, given by the
formula $K(Z,W):= \det (I - ZW^*)^{-1}$, $Z,W \in \mathbb B_{2\times 2}$.

 The kernel $K$ is normalized at $0$ by definition.
For $\delta = (1,0,0,3)$, the matrix
$$\big{(}\!\big(\frac{\partial^\alpha\bar\partial^\beta\log K(0,0)}{\alpha !\beta
!}\big)\!\big{)}_{0\leq\alpha,\beta\leq\delta}$$ is diagonal with
$\frac{\partial_1\partial_4^3\bar\partial_1\bar\partial_4^3\log
K(0,0)}{3!3!} = -1<0$ (in fact for $|\delta|\leq 3$, the
corresponding matrix is diagonal with non-negative entries). Here,
$\delta\geq \mu$ if and only if $\delta_i\geq\mu_i$ for all
$i\in\{1,\ldots,m\}$ and the matrix is written with respect to the
colexicographic ordering. From \cite[Lemma 4.1 and 4.3]{cs}, it
follows that $\log K$ is not positive definite. Hence Theorem
\ref{main} shows that the function $\det (I - ZW^*)^{-t}$ cannot
be positive definite for all $t>0$. Of course, a lot more is
known. Indeed, the set $$\{0 < t: K(Z,W)^t \mbox{ is positive
definite}\}$$ is explicitly determined in \cite[Corollary
4.6]{arazy}.

\bibliographystyle{amsplain}
\bibliography{bibliography}
\end{document}